\documentclass[a4paper,12pt]{amsart}

\usepackage{amsmath,amssymb,amstext}
\usepackage{mathrsfs}
\usepackage{oldgerm}
\usepackage{amsthm}
\usepackage{amsmath}
\usepackage{mathrsfs}
\usepackage[utf8]{inputenc}
\usepackage[mathcal]{euscript}
\usepackage{float}
\usepackage{mathdots}
\usepackage{fullpage}

\usepackage{tikzpeople}
\usepackage[arrow, matrix, curve]{xy}

\usepackage[linesnumbered,ruled,vlined]{algorithm2e}

\usepackage{hyperref}

\newtheorem{theorem}{Theorem}[section]
\newtheorem*{theorem*}{Theorem}
\newtheorem{lem}[theorem]{Lemma}

\theoremstyle{remark}
\newtheorem*{claim}{Claim}
\newtheorem*{rem}{Remark}

\newtheorem*{ex}{Example}
\theoremstyle{definition}
\newtheorem{defi}[theorem]{Definition}
\newtheorem{notation}[theorem]{Notation}

\makeatletter
\newcommand*{\rom}[1]{\expandafter\@slowromancap\romannumeral #1@}
\makeatother

\makeatletter
\newtheorem*{rep@theorem}{\rep@title}
\newcommand{\newreptheorem}[2]{%
\newenvironment{rep#1}[1]{%
 \def\rep@title{#2 \ref{##1}}%
 \begin{rep@theorem}}%
 {\end{rep@theorem}}}
\makeatother

\newreptheorem{theorem}{Theorem}

\DeclareMathOperator{\SL}{SL}

\newcommand{\Z}{\mathbb{Z}}
\newcommand{\N}{\mathbb{N}}

\newcommand{\C}{\mathbb{C}}

\newcommand{\cA}{\mathcal{A}}

\newcommand{\pr}{\mathbb{P}}

\newcommand{\F}{\mathscr{F}}
\newcommand{\B}{\mathscr{B}}
\newcommand{\G}{\mathscr{G}}

\newcommand{\PP}{\mathfrak{P}}
\newcommand{\cP}{\mathcal{P}}

\title{Calculating entries of unitary $\SL_3$-friezes}

\address{Institute of Algebra, Number Theory and discrete Mathematics, Leibniz Universität Hannover, Welfengarten 1, Hannover, DE}

\email{surmann@math.uni-hannover.de}

\author{Lucas Surmann}

\begin{document}

\maketitle

\begin{abstract}
    In this article we consider tame $ \SL_3 $-friezes that arise by specializing a cluster of Plücker variables in the coordinate ring of the Grassmannian $ \G(3,n) $ to $ 1 $. We show how to calculate arbitrary entries of such friezes from the cluster in question. Let $ \F $ be such a cluster. We study the set $ \F_x $ of cluster variables in $ \F $ that share a given index $ x $ and derive a structure Theorem for $ \F_x $. These sets prove central to calculating the first and last non-trivial rows of the frieze. After that, simple recursive formulas can be used to calculate all remaining entries.
\end{abstract}

\section*{Introduction}
Frieze patterns are infinite arrays of numbers with one row of zeroes at the top and bottom and one row of ones below and above, such that each two by two diamond has determinant $ 1 $. They were first introduced in 1971 by Coxeter \cite{Coxeter1971}. Conway and Coxeter then described a correspondence between frieze patterns of positive integers and triangulations of polygons\footnote{Triangulations using only non-crossing lines. We will always use the term triangulation in this sense.} in 1973 \cite{Conway_Coxeter_1973_A}, \cite{Conway_Coxeter_1973_B}. Specifically there is a bijection between integral frieze patterns of width $ w $ and triangulations of a $ (w + 3) $-gon. In 1974, Broline, Crowe and Isaacs further explored this link and found a method to calculate arbitrary entries of frieze patterns from the corresponding triangulation \cite{broline1974geometry}. \\
One can generalize frieze patterns to $ \SL_k $-frieze patterns by requiring that $ k \times k $-diamonds have determinant $ 1 $ and that there are $ k-1 $ rows of zeroes on top and on the bottom. For instance the following is the fundamental region of a tame\footnote{A $ SL_k $-frieze is tame if all $ (k+1) \times (k+1) $-diamonds have determinant $ 0 $.} $ \SL_3 $-frieze pattern:
\[
\xymatrix@C=0.05em@R=.07em{
 0 && 0 && 0 && 0 && 0 && 0 && 0 && 0\\
& 0 && 0 && 0 && 0 && 0 && 0 && 0 && 0\\
&& 1 && 1 && 1 && 1 && 1 && 1 && 1 && 1\\
&&& 4 && 3 && 2 && 5 && 1 && 4 && 5 && 1\\
&&&& 6 && 5 && 4 && 3 && 3 && 7 && 4 && 2\\
&&&&& 9 && 8 && 1 && 8 && 3 && 4 && 7 && 1\\
&&&&&& 13 && 1 && 2 && 6 && 1 && 6 && 2 && 1\\
&&&&&&& 1 && 1 && 1 && 1 && 1 && 1 && 1 && 1\\
&&&&&&&& 0 && 0 && 0 && 0 && 0 && 0 && 0 && 0\\
&&&&&&&&& 0 && 0 && 0 && 0 && 0 && 0 && 0 && 0
}
\]
$ \SL_k $-frieze were introduced by Cordes and Roselle in 1972 \cite{cordes1972generalized}. Classical friezes are then simply $ \SL_2 $-friezes. \\
In the early 2000s cluster algebras were introduced and a connection between frieze patterns and cluster algebras was found. Specifically it was discovered by Fomin and Zelevinsky that the coordinate ring of the Grassmannian $ \G(2,n) $ is a cluster algebra where all cluster variables are Plücker variables and clusters correspond to triangulations of $ n $-gons \cite{Fomin_2003}. Using this, Caldero and Chapoton \cite{caldero2006cluster} showed that all $ \SL_2 $-friezes arise from specializing a given cluster in these cluster algebras to $ 1 $. \\
Scott managed to prove that the coordinate ring of the Grassmannian $ \G(k,n) $ is also a cluster algebra in the case $ k > 2 $ \cite{scott2006grassmannians}. To accomplish this, triangulations of polygons were generalized to the notion of maximal weakly separated families of $ k $-sets. Morier-Genoud, Ovsienko, Schwartz and Tabachnikov then discovered an embedding of the variety of tame $ \SL_k $-friezes into the Grassmannian \cite{morier2014linear}. In 2021 Baur, Faber, Gratz, Serhiyenko, Todorov, Cuntz and Plamondon further expounded on this through a construction called Plücker frieze \cite{baur2021friezes}. If a cluster is specialized to $ 1 $, the resulting frieze will be tame and integral, and $ \SL_k $-friezes that can be derived in this way are called unitary. However the paper also shows that not all tame, integral $ \SL_k $-frieze are unitary. \\
In this paper we focus on the case $ k = 3 $ and unitary friezes. We study the relation between clusters of Plücker variables in the cluster algebra $ \C(\G(3,n)) $ and the $ \SL_3 $-friezes resulting from specializing such clusters to $ 1 $. The goal is to develop a method for calculating arbitrary entries of the frieze with the cluster as the input, similar to Broline, Crowe and Isaacs work on triangulations and $ \SL_2 $-friezes. That being said, our method of calculation is quite different to \cite{broline1974geometry}. \\
In Section 1 we focus on maximal weakly separated families of $ 3 $-subsets of $ [n] $. In the following $ 3 $-sets will be called triangles. Such families $ \F $ of triangles are in bijection with the clusters in $ \C(\G(n,3)) $ that consist of only Plücker variables. We specifically study the set $ \F_x $ of triangles in $ \F $ that contain a given point $ x \in [n] $. This set can be visualised as a graph $ G(\F_x) $, by removing $ x $ from all triangles in $ \F_x $. Through a series of Lemmas we derive a the following Theorem about the possible structure of this graph:
\begin{theorem*}[\ref{thm:triangulation}]
    Let $ \F $ be a maximal family of weakly separated triangles in $ [n] $. Let $ x \in [n] $. Let $ v_1,\dots,v_r $ be those vertices in $ G(\F_x) $ that have degree at least $ 2 $. Then the full subgraph in $ v_1,\dots,v_r $ is a triangulation of an $ r $-gon. Moreover, all other vertices of $ G(\F_x) $ have degree $ 1 $ and their edge connects to one of the $ v_i $.
\end{theorem*}
We also show that certain additional triangles necessarily exist depending on $ \F_x $. Moreover we show that the reverse of the structure theorem also holds, with some prerequisites about the graph.\\
In Section 2 we apply this to the Grassmannian coordinate ring. We describe a series of mutations in $ \F $. These allow us to calculate the value of the Plücker coordinates of the form $ \Delta^{i,i+1,i+3} $ and $ \Delta^{i,i+2,i+3} $. We then use Plücker relations to derive a recursive formula for all other nearly continuous Plücker coordinates\footnote{Nearly continuous Plücker coordinates are Plücker coordinates of the form $ \Delta^{i,i+1,j} $}. \\
Finally in Section $ 3 $ we show how this relates to the entries of tame $ \SL_3 $-friezes that arise from specializing a cluster of Plücker variables to $ 1 $.

\section{Weakly separated families of triangles}

\setcounter{page}{1}

\begin{defi}
    Let $ n \in \N, r \geq 3 $ and let $ (a_1,\dots, a_r) $ be a tuple of pairwise distinct elements of $ [n] = \{1,\dots,n\} $. The elements are called cyclically ordered if they are ordered in the natural ordering of $ \N $ or if there is a $ 2 \leq i \leq r $ such that
    \[
    a_i < a_{i+1} < \dots < a_r < a_1 < a_2 < \dots < a_{i-1}.
    \]
\end{defi}

\begin{defi}
    For $ a, b \in [n] $ with $ a \neq b $ the interval $ (a,b) $ is defined as
    \[
    (a,b) := \{ x \in [n] \; | \; (a,x,b) \text{ is cyclically ordered} \}
    \]
The intervals $ [a,b], (a,b] $ and $ [a,b) $ are defined analogously.
\end{defi}

\begin{defi}[\cite{leclerc1998quasicommuting}]
Let $ A, B \subset [n] $ be subsets. If there are $ a, c \in A \backslash B $ and $ b, d \in B \backslash A $, such that $ (a,b,c) $ and $ (b,c,d) $ are in cyclic order, then $ A $ and $ B $ are said to be crossing. Otherwise they are called weakly separated.
\end{defi}

From now on we only consider subsets $ A \subset [n] $ with $ |A| = 3 $. we call these triangles.

\begin{rem}
Let $ [n] $ be ordered points on a circle. Two triangles $ A, B \subset [n] $ cross if and only if there are lines $ \overline{ab} $ and $ \overline{cd} $ that cross, such that $ a,b \in A \backslash B $ and $ c, d \in B \backslash A $.
\end{rem}

\begin{ex}
Figures 1 and 2 show some constellations in which two triangles do or do not cross.
\end{ex}

\begin{figure}[H]
\begin{tikzpicture}
\coordinate (center) at (0,0);
\coordinate (a) at ($(center) + (100:2cm)$);
\coordinate (b) at ($(center) + (150:2cm)$);
\coordinate (c) at ($(center) + (200:2cm)$);
\coordinate (d) at ($(center) + (250:2cm)$);
\coordinate (e) at ($(center) + (300:2cm)$);
\coordinate (f) at ($(center) + (350:2cm)$);
\draw[dashed] (center) circle(2cm);
\draw (a) node[shift={(100:4pt)}]{a} -- (b) node[shift={(150:4pt)}]{b} 
	-- (c) node[shift={(200:4pt)}]{c} --cycle;
\draw (d) node[shift={(250:4pt)}]{d} -- (e) node[shift={(300:4pt)}]{e}
	-- (f) node[shift={(350:4pt)}]{f} --cycle;

\coordinate (center2) at (5,0);
\coordinate (a) at ($(center2) + (0:2cm)$);
\coordinate (b) at ($(center2) + (60:2cm)$);
\coordinate (c) at ($(center2) + (300:2cm)$);
\coordinate (d) at ($(center2) + (120:2cm)$);
\coordinate (e) at ($(center2) + (240:2cm)$);
\draw[dashed] (center2) circle(2cm);
\draw (a) node[shift={(0:4pt)}]{a} -- (b) node[shift={(60:4pt)}]{b} 
	-- (c) node[shift={(300:4pt)}]{c} --cycle;
\draw (a) -- (d) node[shift={(120:4pt)}]{d}
	-- (e) node[shift={(240:4pt)}]{e} --cycle;

\coordinate (center3) at (10,0);
\coordinate (a) at ($(center3) + (210:2cm)$);
\coordinate (b) at ($(center3) + (330:2cm)$);
\coordinate (c) at ($(center3) + (60:2cm)$);
\coordinate (d) at ($(center3) + (120:2cm)$);
\draw[dashed] (center3) circle(2cm);
\draw (a) node[shift={(210:4pt)}]{a} -- (b) node[shift={(330:4pt)}]{b} 
	-- (c) node[shift={(60:4pt)}]{c} --cycle;
\draw (a) -- (b) -- (d) node[shift={(120:4pt)}]{d} --cycle;
\end{tikzpicture}
\caption{Noncrossing Triangles}
\end{figure}
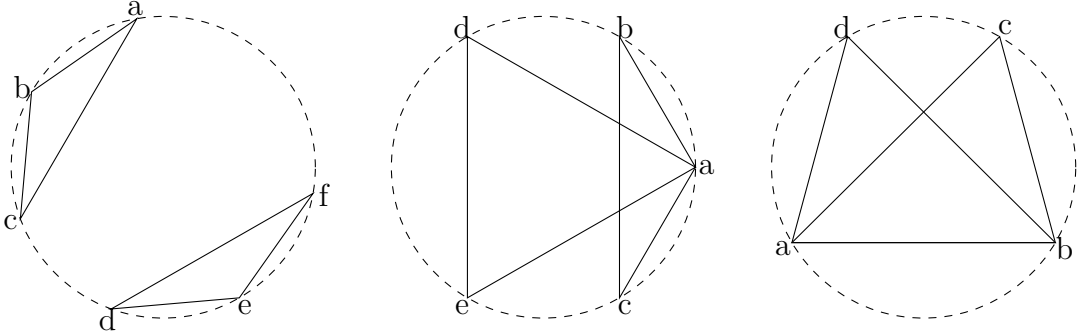

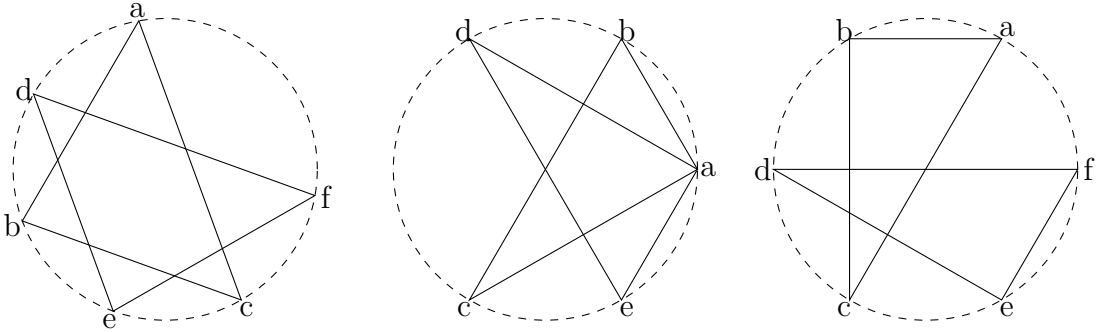
\begin{figure}[H]
\begin{tikzpicture}
\coordinate (center) at (0,0);
\coordinate (a) at ($(center) + (100:2cm)$);
\coordinate (b) at ($(center) + (200:2cm)$);
\coordinate (c) at ($(center) + (300:2cm)$);
\coordinate (d) at ($(center) + (150:2cm)$);
\coordinate (e) at ($(center) + (250:2cm)$);
\coordinate (f) at ($(center) + (350:2cm)$);
\draw[dashed] (center) circle(2cm);
\draw (a) node[shift={(100:4pt)}]{a} -- (b) node[shift={(200:4pt)}]{b} 
	-- (c) node[shift={(300:4pt)}]{c} --cycle;
\draw (d) node[shift={(150:4pt)}]{d} -- (e) node[shift={(250:4pt)}]{e}
	-- (f) node[shift={(350:4pt)}]{f} --cycle;

\coordinate (center2) at (5,0);
\coordinate (a) at ($(center2) + (0:2cm)$);
\coordinate (b) at ($(center2) + (60:2cm)$);
\coordinate (c) at ($(center2) + (240:2cm)$);
\coordinate (d) at ($(center2) + (120:2cm)$);
\coordinate (e) at ($(center2) + (300:2cm)$);
\draw[dashed] (center2) circle(2cm);
\draw (a) node[shift={(0:4pt)}]{a} -- (b) node[shift={(60:4pt)}]{b} 
	-- (c) node[shift={(240:4pt)}]{c} --cycle;
\draw (a) -- (d) node[shift={(120:4pt)}]{d}
	-- (e) node[shift={(300:4pt)}]{e} --cycle;

\coordinate (center3) at (10,0);
\coordinate (a) at ($(center3) + (60:2cm)$);
\coordinate (b) at ($(center3) + (120:2cm)$);
\coordinate (c) at ($(center3) + (240:2cm)$);
\coordinate (d) at ($(center3) + (180:2cm)$);
\coordinate (e) at ($(center3) + (300:2cm)$);
\coordinate (f) at ($(center3) + (0:2cm)$);
\draw[dashed] (center3) circle(2cm);
\draw (a) node[shift={(60:4pt)}]{a} -- (b) node[shift={(120:4pt)}]{b} 
	-- (c) node[shift={(240:4pt)}]{c} --cycle;
\draw (d) node[shift={(180:4pt)}]{d} -- (e) node[shift={(300:4pt)}]{e}
    -- (f) node[shift={(0:4pt)}]{f} --cycle;
\end{tikzpicture}
\caption{Crossing Triangles}
\end{figure}

\begin{lem} \label{lem:cases_crossing}
    Let $ A = \{a,b,c\} $ with $ (a,b,c) $ cyclically ordered and $ B $ be crossing triangles in $ [n] $. Then at least one of the following cases applies.
    \begin{itemize}
        \item[(i)] $ B = \{u,v,w\} $ with $ u \in (a,b), v \in (b,c) $ and $ w \neq b $.
        \item[(ii)] $ B = \{u,v,w\} $ with $ u \in (a,b), w \in (c,a) $ and $ v \neq a $.
        \item [(iii)] $ B = \{u,v,w\} $ with $ v \in (b,c), w \in (c,a) $ and $ u \neq c $.
    \end{itemize}
\end{lem}
\begin{proof}
    Since $ A $ and $ B $ are crossing there are $ d, e \in A \backslash B $ and $ x, y \in B \backslash A $ such that $ (d,x,e) $ and $ (x,e,y) $ are in cyclical order. We have $ (a,b) \cup (b,c) \cup (c,a) = [n] \backslash A $. Since $ x,y \notin A $ this implies that $ x, y \in (a,b) \cup (b,c) \cup (c,a) $. Moreover $ (d,x,e) $ and $ (x,e,y) $ being cyclically ordered implies $ x \in (d,e) $ and $ y \notin (d,e) $. Thus $ x $ and $ y $ are in different ones of the three intervals. \\
    Assume without loss of generality that $ x \in (a,b) $. If $ y \in (b,c) $ we must have $ e = b $. It follows that $ b \notin B $ and we are done. \\ Therefore we can assume $ y \in (c,a) $. If $ d = b $ then $ x \notin (d,e) $ regardless of whether $ e = a $ or $ e = c $. Contradiction. If $ d = c $ then $ y \in (d,e) $. But then $ (x,e,y) $ cannot be cyclically ordered. Hence $ d = a $ which implies $ a \notin B $.
\end{proof}

\begin{defi}[\cite{leclerc1998quasicommuting}]
Let $ \F \subset \PP([n]) $ a set of subsets of $ [n] $. Then $ \F $ is said to be {\em a weakly separated family}, if its elements are pairwise non-crossing. \\
We call $ \F $ a maximal weakly separated family, if there is no weakly separated family $ \F^\prime $ with $ \F \subset \F^\prime $.
\end{defi}

\begin{rem}
 If $ \F $ is a maximal weakly separated family, then every triangle $ A \subset [n] $ that is not in $ \F $ crosses at least one triangle in $ \F $. Otherwise $ A $ could be added to $ \F $ and the family would not be maximal. Hence, if we want to show that a given triangle $ B $ is in a maximal weakly separated family $ \F $, it suffices to show that $ B $ does not cross any of the triangles in $ \F $.
\end{rem}

\begin{defi}
 Let $ \F $ be a weakly separated family of triangles in $ [n] $ and $ x \in [n] $. Then let
 \[
  \F_x := \{C \in \F \: | \: x \in C \}
 \]
 be the subfamily of triangles containing $ x $.
\end{defi}

\begin{defi}
    For $ a \neq b \in [n] \backslash \{x\} $ set
    \[
    a <_x b :\Leftrightarrow (x,a,b) \text{ is cyclically ordered.}
    \]
\end{defi}

$ <_x $ is a total order on $ [n] \backslash \{x\} $.

\begin{defi}
 Let $ A, B \in \F_x $, $ A = \{x,a,b\} $ and $ B = \{x,c,d\} $ for some $ a <_x b $ and $ c <_x d $. We set
 \[
  A \preceq B :\Leftrightarrow a \leq_x c \text{ and } d \leq_x b
 \]
\end{defi}

\begin{defi}
 Let $ \F $ be a weakly separated family of triangles in $ [n] $. Then we define
 \[
  (A,B)_x := \{ C \in \F_x \: | \: A \prec C \prec B \}
 \]
for $ A, B \in \F_x $ with $ A \preceq B $.
\end{defi}

\begin{lem}\label{lem:three_cases}
 Let $ \F $ be a maximal weakly separated family of triangles in $ [n] $ and $ x \in [n] $. Let $ A = \{x,a,b\}, B = \{x,c,d\} \in \F_x $ with $ A \preceq B $, $ a \neq c $ and $ d \neq b $, in other words $ a <_x c $ and $ d <_x b $. If $ C := \{x,a,d\} \notin \F $, then there is a $ D \in \F $ crossing $ C $, fulfilling one of the following cases:
 \begin{itemize}
  \item[(i)] $ D = \{c,y,z\} $ with $ y,z \in (d,b] $,
  \item[(ii)] $ D = \{c,b,z\} $ with $ z \in (b,x) $ or
  \item[(iii)] $ D = \{x,y,z\} $ with $ y \in (a,c] $ and $ z \in (d,b] $.
 \end{itemize}

\begin{figure}[H]
\begin{tikzpicture}
\coordinate (center) at (0,0);
\coordinate (x) at ($(center) + (270:2cm)$);
\coordinate (a) at ($(center) + (210:2cm)$);
\coordinate (b) at ($(center) + (330:2cm)$);
\coordinate (c) at ($(center) + (150:2cm)$);
\coordinate (d) at ($(center) + (30:2cm)$);
\coordinate (y) at ($(center) + (350:2cm)$);
\coordinate (z) at ($(center) + (10:2cm)$);
\draw[dashed] (center) circle(2cm);
\draw[dashed] (x) node[shift={(270:4pt)}]{x} -- (a) node[shift={(210:4pt)}]{a} 
	-- (b) node[shift={(330:4pt)}]{b} --cycle;
\draw[dashed] (x) -- (c) node[shift={(150:4pt)}]{c}
	-- (d) node[shift={(30:4pt)}]{d} --cycle;
\draw[dotted] (x) -- (a)
	-- (d) --cycle;
\draw (c) -- (y) node[shift={(350:4pt)}]{y}
	-- (z) node[shift={(10:4pt)}]{z} --cycle;
	
\coordinate (center2) at (5,0);
\coordinate (x) at ($(center2) + (270:2cm)$);
\coordinate (a) at ($(center2) + (210:2cm)$);
\coordinate (b) at ($(center2) + (330:2cm)$);
\coordinate (c) at ($(center2) + (150:2cm)$);
\coordinate (d) at ($(center2) + (30:2cm)$);
\coordinate (z) at ($(center2) + (300:2cm)$);
\draw[dashed] (center2) circle(2cm);
\draw[dashed] (x) node[shift={(270:4pt)}]{x} -- (a) node[shift={(210:4pt)}]{a} 
	-- (b) node[shift={(330:4pt)}]{b} --cycle;
\draw[dashed] (x) -- (c) node[shift={(150:4pt)}]{c}
	-- (d) node[shift={(30:4pt)}]{d} --cycle;
\draw[dotted] (x) -- (a)
	-- (d) --cycle;
\draw (c) -- (b) -- (z) node[shift={(300:4pt)}]{z} --cycle;

\coordinate (center3) at (10,0);
\coordinate (x) at ($(center3) + (270:2cm)$);
\coordinate (a) at ($(center3) + (210:2cm)$);
\coordinate (b) at ($(center3) + (330:2cm)$);
\coordinate (c) at ($(center3) + (150:2cm)$);
\coordinate (d) at ($(center3) + (30:2cm)$);
\coordinate (y) at ($(center3) + (180:2cm)$);
\coordinate (z) at ($(center3) + (0:2cm)$);
\draw[dashed] (center3) circle(2cm);
\draw[dashed] (x) node[shift={(270:4pt)}]{x} -- (a) node[shift={(210:4pt)}]{a} 
	-- (b) node[shift={(330:4pt)}]{b} --cycle;
\draw[dashed] (x) -- (c) node[shift={(150:4pt)}]{c}
	-- (d) node[shift={(30:4pt)}]{d} --cycle;
\draw[dotted] (x) -- (a)
	-- (d) --cycle;
\draw (x) -- (y) node[shift={(180:4pt)}]{y}
	-- (z) node[shift={(0:4pt)}]{z} --cycle;
\end{tikzpicture}
\caption{Cases (i), (ii), (iii)}
\end{figure}
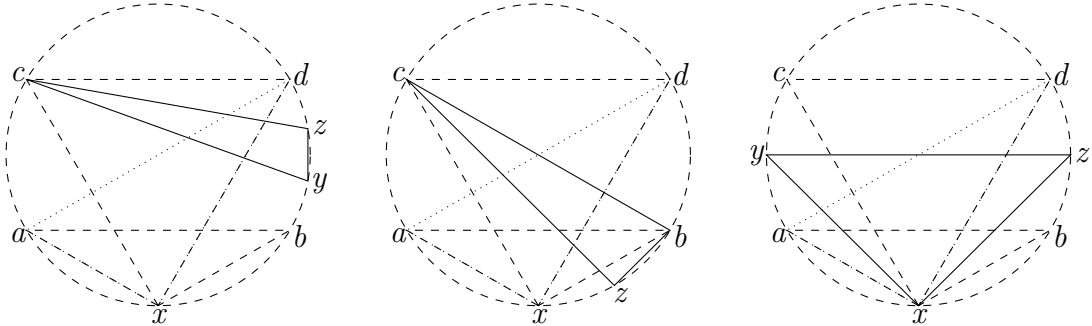

\end{lem}
\begin{proof}
 As $ \F $ is maximal and $ C \notin \F $, there is a triangle $ D = \{u,v,w\} \in \F $  that crosses $ C $. Otherwise $ C $ could be added to $ \F $ and $ \F $ would not be maximal. By Lemma \ref{lem:cases_crossing} one of the following cases must apply: \\
 
 Case 1: $ u \in (x,a), v \in (a,d) $ and $ w \neq a $. \\
 But then $ \overline{xa} $ and $ \overline{ab} $ cross $ \overline{uv} $ with $ u,v \notin A $ and $ a \notin D $. This means that $ A $ crosses $ D $, unless $ x \in D $  and $ b \in D $. But then we would have $ w = x $ and $ w = b $. Contradiction.\\
 
 Case 2: $ u \in (x,a), w \in (d,x) $ and $ v \neq x $. \\
 Then $ \overline{xc} $ and $ \overline{xd} $ cross $ \overline{uw} $ with $ u,w \notin B $ and $ x \notin D $. Hence $ B $ and $ D $ cross unless $ v = c $ and $ v = d $. Contradiction. \\
 
 Case 3: $ v \in (a,d), w \in (d,x) $ and $ u \neq d $. \\
 First assume $ w \in (b,x) $. Then $ \overline{ab} $ and $ \overline{bx} $ cross $ \overline{vw} $ with $ v,w \notin A $. Therefore we have $ u = a $ or $ u = b $ and also $ u = b $ or $ u = x $. Thus $ u = b $.\\
 However in this case $ \overline{dx} $ crosses $ \overline{vw} $ with $ w \notin B $ and $ d,x \notin D $. Since $ B $ and $ D $ do not cross, this implies $ v \in B $, i.E. $ v = c $. \\
 In conclusion, $ D = \{c,b,w\} $ with $ w \in (b,x) $ as in (ii).\\
 
 Next, assume $ w = b $. Then $ \overline{dx} $ crosses $ \overline{vw} $ with $ w \notin B $ and $ d \notin D $. It follows that $ v = c $ or $ u = x $. \\
 Let $ u = x $. Then, if $ v \in (a,c] $ we have $ D = \{x,v,b\} $ and we are in (iii). \\
 On the other hand if $ v \in (c,d) $, then $ \overline{cd} $ crosses $ \overline{vw} $ with $ c,d \notin D $ and $ v,w \notin B $. Contradiction. \\
 We can now assume $ u \neq x $ and hence $ v = c $. If $ u \in (x,d) $, then $ \overline{uw} $ crosses $ \overline{xd} $ with $ d,x \notin D $ and $ u,w \notin B $. Contradiction. It follows that $ D = \{c,b,u\} $ with $ u \in (d,x) $. Depending on $ u $ this is (i) or (ii). \\
 
 Finally, assume $ w \in (d,b) $. Then $ \overline{xd} $ crosses $ \overline{vw} $ with $ w \notin B $ and $ d \notin D $. Then $ u = x $ or $ v = c $. \\
 Let $ u \neq x $ and therefore $ v = c $. Now if $ u \in (b,a) $, then $ \overline{ab} $ crosses $ \overline{vu} $ with $ a,b \notin D $ and $ v,u \notin A $. Contradiction. 
 If $ u \in (x,d) $, then $ \overline{xd} $ crosses $ \overline{uw} $ with $ d,x \notin D $ and $ u,w \notin B $, which is also a contradiction. \\
 Hence $ u \in (d,b] $. In this case we have $ D = \{c,w,u\} $ with $ w, u \in (d,b] $, which is (i). \\
 Now we can assume $ u = x $. Now, if $ v \in (c,d) $, then $ \overline{cd} $ crosses $ \overline{vw} $ with $ c,d \notin D $ and $ v,w \notin B $. Contradiction. Thus $ v \in (a,c] $. But then we are in (iii).

\end{proof}

\begin{rem}
 By symmetry, Lemma \ref{lem:three_cases} implies the following: If $ C := \{x,c,b\} \notin \F $ in the situation of the lemma, then there is a $ D \in \F $ crossing $ C $ with
 \begin{itemize}
  \item[(i)] $ D = \{y,z,d\} $ with $ y,z \in [a,c) $,
  \item[(ii)] $ D = \{z,a,d\} $ with $ z \in (x,a) $ or
  \item[(iii)] $ D = \{x,y,z\} $ with $ y \in [a,c) $ and $ z \in [d,b) $.
 \end{itemize}

\end{rem}

\begin{lem}\label{lem:empty_interval}

 Let $ \F $ be a maximal weakly separated family of triangles in $ [n] $ and $ x \in [n] $. Let $ A, B \in \F_x $ with $ A \preceq B $ and $ (A,B)_x = \emptyset $. Then $ |A \cap B| \geq 2 $. If we write $ A = \{x,a,b\} $ and $ B = \{x,c,d\} $ with $ a \leq_x b $ and $ c \leq_x d $, this implies $ b = c $ or $ d = b $.
\end{lem}

\begin{proof}
 Assume $ b \neq c $ and $ d \neq b $. In particular $ C := \{x,a,d\}, C^\prime := \{x,c,b\} \notin \{A, B\} $. If $ C \in \F $, then we would have $ C \in (A,B)_x $. Thus $ C \notin \F $. By Lemma \ref{lem:three_cases} we have $ D = \{c,y,z\} \in \F $ with $ y,z \in (d,b] $ or $ D = \{c,b,z\} \in \F $ with $ z \in (b,x) $ (the case (iii) would contradict $ (A,B)_x = \emptyset $). In either case $ D = \{c,y,z\} $ with $ y,z \in (d,x) $.\\
 Similarily, if $ C^\prime \in \F $, we would have $ C^\prime \in (A,B)_x $. By the last remark, there then must be $ D^\prime = \{v,w,d\} \in \F $ with $ v,w \in [a,c) $ or $ D^\prime = \{v,a,d\} \in \F $ with $ v \in (x,a) $. Again, (iii) is ruled out by $ (A,B)_x = \emptyset $. In both possible cases $ D^\prime = \{v,w,d\} $ with $ v,w \in (x,c) $.\\
 Then $ \overline{cy} $ crosses $ \overline{vd} $ with $ c,y \notin D^\prime $ and $ v,d \notin D $. Hence $ D $ crosses $ D^\prime $ in contradiction to both being in $ \F $. Our assumption must be wrong, accordingly $ b = c $ or $ d = b $.
\end{proof}

\begin{lem} \label{lem:new_triangle}
 Let $ \F $ be a maximal weakly separated family of triangles in $ [n] $, $ x \in [n] $ and $ A = \{x,a,b\}, B = \{x,c,d\} \in \F $ with $ a <_x c <_x d <_x b $. Then there is a point $ y \in (a,c] \cup [d,b) $ such that $ \{x,a,y\} $ and $ \{x,y,b\} $ are both in $ \F $.

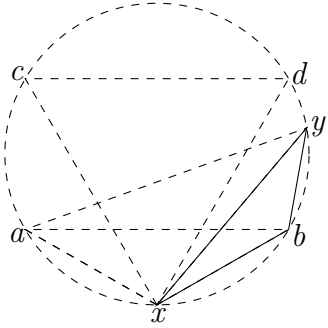
\begin{figure}[H] \label{fig:triangle_lemma}
\centering
\begin{tikzpicture}
\coordinate (center) at (0,0);
\coordinate (x) at ($(center) + (270:2cm)$);
\coordinate (a) at ($(center) + (210:2cm)$);
\coordinate (b) at ($(center) + (330:2cm)$);
\coordinate (c) at ($(center) + (150:2cm)$);
\coordinate (d) at ($(center) + (30:2cm)$);
\coordinate (y) at ($(center) + (10:2cm)$);
\draw[dashed] (center) circle(2cm);
\draw[dashed] (x) node[shift={(270:4pt)}]{x} -- (a) node[shift={(210:4pt)}]{a} 
	-- (b) node[shift={(330:4pt)}]{b} --cycle;
\draw[dashed] (x) -- (c) node[shift={(150:4pt)}]{c}
	-- (d) node[shift={(30:4pt)}]{d} --cycle;
\draw (x) -- (b)
	-- (y) node[shift={(10:4pt)}]{y} --cycle;
\draw[dashed] (x) -- (a) 
	-- (y) --cycle;
\end{tikzpicture}
\caption{Illustration of the Lemma}
\end{figure}

\end{lem}

\begin{proof}
 Clearly we have $ A \prec B $. We also know $ (A,B)_x \neq \emptyset $, since otherwise Lemma \ref{lem:empty_interval} would imply $ a = c $ or $ d = b $, contradicting our assumption. Let $ C = \{x,e,f\} \in (A,B)_x $ with $ e <_x f $. If $ a \neq e $ and $ f \neq b $, then we can replace $ B $ with $ C $. Thus we can assume that either $ a = e $ or $ f = b $ for all such triangles $ C $. By symmetry, it is enough to only consider the case $ a = e $. Additionally we can assume that $ f $ is maximal, i.e. there is no triangle $ \{x,a,f^\prime\} \in \F $ with $ f <_x f^\prime <_x b $.
 \begin{claim}
  $ T := \{x,f,b\} \in \F $.
 \end{claim}
 Assume that $ T \notin \F $. Then there is a triangle $ D = \{u,v,w\} \in \F $ crossing $ T $, since $ \F $ is maximal weakly separated. Then one of the following cases must apply: \\
 
 Case 1: $ u \in (x,f), v \in (f,b) $ and $ w \neq f $.\\
 Then $ \overline{uv} $ crosses $ \overline{xf} $ with $ v \notin C $ and $ f \notin D $. If $ u \notin C $ and $ x \notin D $, then $ C $ would cross $ D $ which cannot be the case. Hence we have $ u = a $ or $ w = x $. \\
 Now assume $ w \neq x $ and thus $ u = a $. Then $ D $ and $ B $ cross in $ \overline{uv} $ and $ \overline{xc} $, except if $ w = c $. But if that is the case, then $ D $ and $ C $ cross in $ \overline{wv} $ and $ \overline{xf} $. \\
 It follows that $ w = x $. Now, if $ u \in (x,a) $, then $ D $ crosses $ A $ in $ \overline{uv} $ and $ \overline{ab} $. Contradiction. \\
 If $ u \in (a,f) $, then $ D $ crosses $ C $ in $ \overline{uv} $ and $ \overline{af} $. Contradiction. \\
 Therefore $ u = a $. But then $ D = \{x,a,v\} \in \F $ with $ f <_x v <_x b $, which contradicts our assumption that $ f $ is maximal. \\
 
 Case 2: $ u \in (x,f) , w \in (b,x) $ and $ v \neq x $. \\
 Then $ D $ crosses $ A $ in $ \overline{uw} $ and $ \overline{xb} $, except if $ u = a $ or $ v = b $. $ D $ also crosses $ C $ in $ \overline{uw} $ and $ \overline{xf} $, except $ u = a $ or $ v = f $. Since $ v = b $ and $ v = f $ cannot both be true, we have $ u = a $. \\
 But then $ D $ crosses $ B $ in $ \overline{uw} $ and $ \overline{xc} $ as well as $ \overline{xd} $, except if $ v = c $ and $ v = d $. Since $ c \neq d $, this is a contadiction. \\
 
 Case 3: $ v \in (f,b), w \in (b,x) $ and $ u \neq b $. \\
 Then $ D $ crosses $ A $ in $ \overline{vw} $ and $ \overline{xb} $ as well as $ \overline{ab} $, except if $ u = x $ and $ u = a $. Contradiction. \\
 
 Since none of the cases apply, $ D $ cannot exist. But then the claim is true and $ \{x,f,b\} \in \F $. By setting $ y := f $ we have $ \{x,a,y\}, \{x,y,b\} \in \F $.

\end{proof}

\begin{defi}
 Let $ G(\F_x) $ be the graph where the vertices are the elements of $ [n] $ that appear in $ \F_x $ apart from $ x $ and where there is an edge between $ a $ and $ b $ for every triangle $ \{x,a,b\} \in \F_x $.
\end{defi}

\begin{defi}
    Let $ G = (V,E) $ be a graph equipped with a cyclical order on its vertices. We call $ G $ a {\em triangulation of a polygon}, if
    \begin{itemize}
        \item[(i)] neighbouring vertices in the order are connected by an edge,
        \item[(ii)] no two edges cross each other and
        \item[(iii)] all inner regions of the graph are triangles.
    \end{itemize}
\end{defi}

It is easy to see that triangulations with $ V = [n] $ are exactly the maximal weakly separated families of $ 2 $-subsets of $ [n] $.

\begin{theorem} \label{thm:triangulation}
Let $ \F $ be a maximal family of weakly separated triangles in $ [n] $. Let $ x \in [n] $. Let $ v_1,\dots,v_r $ be those vertices in $ G(\F_x) $ that have degree at least $ 2 $. Then the full subgraph in $ v_1,\dots,v_r $ is a triangulation of an $ r $-gon. Moreover, all other vertices of $ G(\F_x) $ have degree $ 1 $ and their edge connects to one of the $ v_i $.
\end{theorem}

\begin{ex}

The figure below shows an example of how $ G(\F_x) $ could look according to the Theorem.

\begin{figure}[H] \label{fig:g}
\centering
\begin{tikzpicture}
\coordinate (center) at (0,0);
\coordinate (xp) at ($(center) + (255:2cm)$);
\coordinate (xm) at ($(center) + (285:2cm)$);
\coordinate (xpp) at ($(center) + (225:2cm)$);
\coordinate (xmm) at ($(center) + (315:2cm)$);
\coordinate (a) at ($(center) + (150:2cm)$);
\coordinate (b) at ($(center) + (80:2cm)$);
\coordinate (c) at ($(center) + (115:2cm)$);
\coordinate (d) at ($(center) + (20:2cm)$);
\coordinate (l2) at ($(center) + (205:2cm)$);
\coordinate (l3) at ($(center) + (185:2cm)$);
\coordinate (l4) at ($(center) + (50:2cm)$);
\coordinate (l5) at ($(center) + (350:2cm)$);

\draw[dashed] (center) circle(2cm);
\draw (xm) node[shift={(285:5pt)}]{$v_7$} -- (xp) node[shift={(255:5pt)}]{$v_1$};
\draw (xp) -- (xpp) node[shift={(225:5pt)}]{$l_1$};
\draw (xm) -- (xmm) node[shift={(315:5pt)}]{$v_6$};
\draw (xp) -- (a) node[shift={(150:5pt)}]{$v_2$};
\draw (xm) -- (a);
\draw (xm) -- (b) node[shift={(80:5pt)}]{$v_4$};
\draw (a) -- (b);
\draw (b) -- (xmm);
\draw (a) -- (c) node[shift={(115:5pt)}]{$v_3$};
\draw (c) -- (b);
\draw (b) -- (d) node[shift={(20:5pt)}]{$v_5$};
\draw (d) -- (xmm);
\draw (a) -- (l2) node[shift={(205:5pt)}]{$l_2$};
\draw (a) -- (l3) node[shift={(185:5pt)}]{$l_3$};
\draw (d) -- (l4) node[shift={(50:5pt)}]{$l_4$};
\draw (d) -- (l5) node[shift={(350:5pt)}]{$l_5$};
\end{tikzpicture}
\caption{Possible $ G(\F_x) $}
\end{figure}

\end{ex}

\begin{defi}
 We call the vertices $ v_i $ from Theorem \ref{thm:triangulation} the triangulation points of $ \F_x $. The vertices $ v $ in $ G(\F_x) $ with $ deg(v) = 1 $ that are connected to $ v_i $ by their edge are called leaves in $ v_i $.
\end{defi}

\begin{rem}
The continuous triangles $ \{x-2,x-1,x\}, \{x-1,x,x+1\}, \{x,x+1,x+2\} $ do not cross any triangle and are thus contained in all maximal weakly separated families. And since they contain $ x $, they are also elements of $ \F_x $. Therefore $ x-1 $ and $ x+1 $ are both triangulation points,  $ v_1 = x+1 $ and $ v_r = x-1 $. $ x+2 $ and $ x-2 $ will likewise always be vertices in $ G(\F_x) $, but they can either be leaves or triangulation points.
\end{rem}

\begin{proof}[Proof of Theorem \ref{thm:triangulation}]
Let $ r $ be the number of triangulation points in $ G(\F_x) $, i.e. the number of vertices of order $ \geq 2 $. We have $ r \geq 2 $ since $ x-1 $ and $ x+1 $ both have order $ \geq 2 $. We will show that for every $ 2 \leq k \leq r $ there are $ k $ triangulation points including $ x+1 $ and $ x-1 $ such that the complete subgraph in these vertices is a triangulation of an $ k $-gon. The first part of the Theorem follows with $ k = r $. Induction by $ k $: \\

$ k = 2 $: A triangulation of a $ 2 $-gon is just $ 2 $ points with an edge in between. Choose $ v_1 = x+1 $ and $ v_r = x-1 $. There is an edge between them since $ \{x-1,x,x+1\} \in \F $. \\

$ k \mapsto k + 1 , k < r $: Let $ T_k = \{v_{i_1},\dots,v_{i_k}\} $ be the set of triangulation points from the previous step. We can assume that this set is ordered with respect to $ <_x $. Let $ v_j $ be any triangulation point not in $ T_k $. Choose the largest $ 1 \leq l \leq k $ with $ v_{i_l} <_x v_j $. This exists, since $ v_{i_1} = x+1 \in T_k $ and $ v_{i_1} <_x v_j $. We also know $ l < k $, since $ v_{i_k} = x-1 $  and $ v_j <_x x-1 $. Thus $ v_j \in (v_{i_l},v_{i_{l+1}}) $. \\
Furthermore, $ v_j $ has degree at least $ 2 $ in $ G(\F_x) $, so there are points $ P, Q \in [n] $ with $ \{x,v_j,P\}, \{x,v_j,Q\} \in \F $. If $ \{P, Q\} = \{v_{i_l},v_{i_{l+1}}\} $ then the complete subgraph in $ T_k \cup \{v_j\} $ is a triangulation of a $ k+1 $-gon. \\
Otherwise assume, without loss of generality, $ P \notin \{v_{i_l},v_{i_{l+1}}\} $.  In this situation we have $ P \in (v_{i_l},v_{i_{l+1}}) $ because otherwise $ \{x,v_j,P\} \in \F $ would cross $ \{x,v_{i_l},v_{i_{l+1}}\} \in \F $. \\
It follows that we are in the situation of Lemma \ref{lem:new_triangle} with $ A = \{x,v_{i_l},v_{i_{l+1}}\} $ and $ B = \{x,v_j,P\} $. Therefore $y \in (v_{i_l},v_{i_{l+1}}) $ exists with $ \{x,v_{i_l},y\}, \{x,y,v_{i_{l+1}}\} \in \F $. In particular $ y $ has degree $ \geq 2 $ in $ \F_x $ and is thus a triangulation point. Then the complete subgraph in $ T_k \cup \{y\} $ is a triangulation of a $ k+1 $-gon. \\

It remains to show that every other vertex has degree $ 1 $ and is adjacent to a triangulation point. Let $ Q_1 $ be a vertex with degree $ \leq 1 $. $ Q_1 $ must have at least one edge by definition of $ G(\F_x) $ and is thus a leaf. Let $ Q_2 $ be its sole neighbour. Assume that $ Q_2 $ is also a leaf. Choose the largest $ 1 \leq l \leq r $ with $ v_l <_x Q_1 $. The existence of $ l $ and $ l < r $ follow with the same argument as before. Then we have $ Q_1, Q_2 \in (v_l,v_{l+1}) $, because otherwise $ \{x,Q_1,Q_2\} \in \F $ would cross $ \{x,v_l,v_{l+1}\} \in \F $. We can apply Lemma \ref{lem:new_triangle}, which yields a $ y \in (v_l,v_{l+1}) $ with $ \{x,v_l,y\}, \{x,y,v_{l+1}\} \in \F $. Hence $ y $ is a triangulation point. But $ y \in (v_l,v_{l+1}) $ implies $ y \notin \{v_1,\dots,v_k\} $. Since $ \{v_1,\dots,v_r\} $ is the set of all triangulation points in $ G(\F_x) $, this is a contradiction. It follows that $ Q_2 $ is a triangulation point.
\end{proof}

\begin{lem} \label{lem:leaf_locations}
Let $ P \in [n] $ be a leaf in $ \F_x $. Let $ v_i \in [n] $ be the triangulation point adjacent to $ P $. Then we have
\[
P \in \begin{cases} (v_1,v_2), & \text{if } i = 1 \\
				    (v_{i-1},v_{i+1}), & \text{if } 1 < i < r \\
					(v_{r-1},v_r), & \text{if } i = r
\end{cases}
\]
\end{lem}
\begin{proof}
Let $ v_j $ be the maximal triangulation point with $ v_j <_x P $. Such a triangulation point exists since $ x $ is not a vertex in $ \F_x $ and since $ x + 1 $ is a triangulation point and therefore $ v_1 <_x P $. We also have $ j < r $, since $ v_r >_x P $. It follows that $ P \in (v_j, v_{j+1}) $. If $ j = i $ or $ i = j+1 $ we are done. \\
Otherwise $ v_i \notin (v_j,v_{j+1}) $. But $ v_j $ and $ v_{j+1} $ are two neighbouring triangulation points and by Theorem \ref{thm:triangulation} there is an edge between $ v_j $ and $ v_{j+1} $ in $ G(\F_x) $. Thus $ \{x,v_j,v_{j+1}\} \in \F $ and this triangle crosses $ \{x,P,v_i \} \in \F $. Contradiction.
\end{proof}

\begin{lem} \label{lem:border_triangles}
Let $ \F $ be a maximal family of weakly separated triangles in $ [n] $ and let $ B := v_i \in [n] $ be a triangulation point. Let $ P_1,...P_s $ be the leaves in $ B $, ordered. Let $ P_0 $ be the preceeding triangulation point $ v_{i-1} $ and $ P_{s+1} $ the next triangulation point $ v_{i+1} $ (with the conventions $ v_0 = v_r $ and $ v_{r+1} = v_1 $). Then the triangles $ \{B,P_j,P_{j+1}\} $ for $ 0 \leq j \leq s $ are all elements of $ \F $, except possibly if $ B = v_1 $ and $ j = 0 $ or $ B = v_r $ and $ j = s $.
\end{lem}

\begin{proof}
 Let $ t = \{B,P_j,P_{j+1}\} $ be such a triangle. Let $ s = \{u,v,w\} \in \F $ be a triangle that crosses $ t $. First we assume that $ B \in (P_j,P_{j+1}) $. \\
 
 Case 1: $ u \in (P_j,B), v \in (B,P_{j+1}) $ and $ w \neq B $. \\
 First assume $ u = x $. This implies $ x \in (P_j,B) $. In particular, $ x \in (P_0,B) = (v_{i-1},v_i) $. But this is only possible if $ i = 1 $, i.E. $ B = x+1 $ and $ P_0 = x-1 $. Now $ B \in (P_j,P_{j+1}) $ implies $ j = 0 $. This case is excluded in the statement of the Theorem. Similarly $ v = x $ would lead to $ B = v_r $ and $ j = s $. \\
 Therefore we can assume $ u \neq x $ and $ v \neq x $. Then $ \overline{uv} $ crosses $ \{x,P_j,B\} $ in $ \overline{P_jB} $ and $ \{x,B,P_{j+1}\} $ in $ \overline{BP_{j+1}} $, unless $ w = P_j $ and $ w = P_{j+1} $. Since both cannot be true at the same time, we have a contradiction. \\
 
 Case 2: $ u \in (P_j,B) $, $ w \in (P_{j+1},P_j) $ and $ v \neq P_j $. \\
 Then $ \overline{uw} $ crosses $ \{x,P_j,B\} \in \F $ in $ \overline{P_jB} $, unless $ v = B $ or $ u = x $ or $ w = x $.
 The possibility $ u = x $ can be ruled out like in case 1. \\
 If $ w = x $, then $ \overline{uv} $ is an edge in $ G(\F_x) $. Accordingly, $ u $ is a vertex of this graph and either a triangulation point or a leaf. But $ u \in (P_0,B) $ cannot be a triangulation point since there are no triangulation points between $ P_0 = v_{i-1} $ and $ B = v_i $. And $ u \in (P_j,P_{j+1}) $ cannot be a leaf in $ B $. Therefore $ u $ is a leaf in another triangulation point $ v = v_k $. But if $ v_k < P_j $ then the triangles $ \{x,v_k,u\} \in \F $ and $ \{x,B,P_j\} \in \F $ cross. If $ v_k >_x P_j $ then $ v_k \geq v_{i+1} $. In this case $ \{x, u, v_k\} \in \F $ crosses $ \{x,v_{i-1},B\} \in \F $. The case $ v = P_j $ is ruled out by our assumption $ v \neq P_j $. \\
 Therefore we have $ w \neq x $ and $ v = B $. $ w \in (P_{j+1},P_j) $ and $ w \neq x $ implies $ w \in (P_{j+1},x) $ or $ w \in (x,P_j) $. In the former case $ \overline{uw} $ crosses $ \{x,B,P_j+1\} \in \F $ in $ \overline{xP_{j+1}} $. In the latter case $ \overline{uw} $ crosses $ \{x,P_j,B\} \in \F $ in $ \overline{xP_j} $. Contradiction. \\
 
 Case 3: $ v \in (B,P_{j+1}) $, $ w \in (P_{j+1}, P_j) $ and $ u \neq P_{j+1} $. \\
 This case is symmetric to case 2 (with the exception that $ v = x $ is contradicted as in case $ 1 $ instead of $ u = x $). \\
 
 Now assume $ P_{j+1} \in (P_j,B) $ (the case $ P_j \in (B,P_{j+1}) $ is symmetric to this, hence an analogous argument applies). \\
 
 Case 1: $ u \in (P_j,P_{j+1}) $, $ v \in (P_{j+1},B) $ and $ w \neq P_{j+1} $. \\
 Here we have $ u \neq x $, because otherwise $ B = x+1 $ and $ P_0 = x-1 $, which would imply $ P_{j+1} \in (x,x+1) $. We also have $ v \neq x $. Therefore $ \overline{uv} $ crosses $ \{x,P_{j+1},B\} \in \F $ in $ \overline{P_{j+1}B} $ and $ \overline{P_{j+1}x} $ unless $ w = B $ and $ w = x $, which cannot both be true. Contradiction. \\
 
 Case 2: $ u \in (P_j, P_{j+1}) $, $ w \in (B,P_j) $ and $ v \neq P_j $. \\
 Then $ \overline{uw} $ crosses $ \{x,P_j,B\} \in \F $ in $ \overline{P_jB} $, unless $ v = B $ or $ w = x $ ($ u = x $ is impossible like in case 1). If $ w = x $, then $ \overline{uv} $ is an edge in $ G(\F_x) $, but just like before $ u $ cannot be a vertex in this graph. \\
 Hence $ w \neq x $ and $ v = B $. Then $ \overline{uw} $ crosses $ \{x,P_j,B\} \in \F $ in $ \overline{xP_j} $ or $ \{x,P_{j+1},B\} \in \F $ in $ \overline{xP_{j+1}} $ depending on whether $ w \in (x,P_j) $ or $ w \in (B,x) $. \\
 
 Case 3: $ v \in (P_{j+1},B) $, $ w \in (B,P_j) $ and $ u \neq B $. \\
 Then $ \overline{vw} $ crosses $ \{x,P_j,B\} \in \F $ in $ \overline{P_jB} $, unless $ w = x $ or $ u = P_j $. If $ w = x $, then $ \overline{uv} $ is an edge in $ G(\F_x) $. Hence $ v $ is a vertex in this graph. Since $ v \in (P_0,B) $, $ v $ cannot be a triangulation point and is thus a leaf. By Lemma \ref{lem:leaf_locations} this leaf is connected to $ P_0 $ or $ B $. If it is connected to $ P_0 $, then $ \{x,P_0,v\} \in \F $ and this crosses $ \{x,P_{j+1},B\} \in \F $. 
 We can conclude that $ v $ is a leaf connected to $ B $. Then $ \overline{vu} $ must be the one edge incident to $ v $ and $ u = B $ must be its other vertex. This contradicts our assumption $ u \neq B $. \\
 It follows that $ w \neq x $ and thus $ u = P_j $. This implies that $ \overline{uv} $ crosses $ \{x,P_{j+1},B\} \in \F $ in $ \overline{P_{j+1}B} $. Contradiction. \\
 
 Since $ s $ does not exist, it follows that $ t = \{B,P_j,P_{j+1}\} \in \F $.
\end{proof}

\begin{defi}
 We call the triangles in Lemma \ref{lem:border_triangles} border triangles of $ \F_x $.
\end{defi}

We can also show the inverse of Theorem \ref{thm:triangulation}:

\begin{theorem}

Let $ n \in \N $ and $ x \in [n] $. Let $ G $ be a graph such that the vertices are a subset of $ [n] \backslash \{x\} $. Let $ T $ be a complete subgraph of $ G $. Assume that the following conditions hold:

\begin{itemize}

	\item[(i)] $ T $ is a triangulation of a polygon and $ x-1 $ and $ x+1 $ are vertices in $ T $.
	\item[(ii)] All vertices in $ G $ that are not in $ T $ have degree $ 1 $.
	\item[(iii)] Every vertex $ l $ of degree $ 1 $ is either adjacent to the maximal vertex $ v \in T $ with $ v <_x l $ or the minimal vertex $ v \in T $ with $ l <_x v $.
	\item[(iv)] Let $ t_1 $ and $ t_2 $ be vertices of $ T $ with $ t_1 <_x t_2 $. If $ l_1 $ is a leaf in $ G $ adjacent to $ t_1 $ and $ l_2 $ a leaf adjacent to $ t_2 $, then $ l_1 <_x l_2 $.
	\item[(v)] $ x-2 $ and $ x+2 $ are vertices in $ G $ and the edges $ \{x-2,x-1\} $ and $ \{x+1,x+2\} $ exist in $ G $.

\end{itemize}

Then there is a maximal weakly separated family $ \F $ of triangles in $ [n] $ such that $ G(\F_x) = G $.

\end{theorem}

\begin{proof}
 Let $ \F $ be the family of triangles $ \{x,a,b\} $ where $ \{a,b\} $ is an edge in $ G $. Also, let $ P_0,...,P_s $ be the vertices of $ T $ in cyclical order with $ P_0 = x+1 $ and $ P_s = x-1 $. For $ i = 0,\dots,s $ let $ Q_{i,1},\dots,Q_{i,r_i-1} $ be the leaves in $ G $ adjacent to $ P_i $ ($ r_i = 1 $ is possible). We use the convention $ Q_{i,0} = P_{i-1} $ and $ Q_{i,r_i} = P_{i+1} $ with $ P_{-1} = P_s $ and $ P_{s+1} = P_0 $. Then let $ \B $ be the family of triangles $ \{P_i,Q_{i,j},Q_{i,j+1}\} $ for $ i = 0,\dots,s $ and $ j = 0,\dots,r_i-1 $, except in the cases  $ i = j = 0 $ and $ i= s, j = r_s - 1 $. \\
 Step 1: $ \F \cup \B $ is a weakly separated family of triangles.\\
 Assume there are two crossing triangles $ \{x,a,b\}, \{x,c,d\} \in \F $. Then the line $ \overline{ab} $ must cross $ \overline{cd} $, since both triangles contain $ x $. In particular, $ \{a,b\} $ and $ \{c,d\} $ cannot both be edges in $ T $, since the lines in the triangulation of a polygon never cross. Without loss of generality we can assume $ a <_x c <_x b <_x d $. \\
 Now assume that the edge $ \{a,b\} $ is not in $ T $, but $ \{c,d\} $ is. The fact that $ T $ is a complete subgraph then implies that $ a $ or $ b $ is not in $ T $. Then $ a $ or $ b $ is a leaf. If $ a $ is a leaf then $ b $ is the one vertex in $ T $ adjacent to $ a $. We have $ a <_x b $, but $ c $ is another vertex in $ T $ with $ a <_x c <_x b $. Clearly $ b $ is not minimal in $ T $ with this property, which contradicts $ (iii) $. If $ b $ is the leaf, we have $ a <_x b $, but $  $ is not maximal, since $ a <_x c <_x b $. Again this contradicts $ (iii) $. \\
 The case that $ \{a,b\} $ is in $ T $ while $ \{c,d\} $ is not follows analogously. \\
 We can now assume that neither $ \{a,b\} $ nor $ \{c,d\} $ are edges in $ T $. We no longer assume $ a <_x c <_x b <_x d $, instead we assume that $ a $ is the leaf in $ \{a,b\} $ and $ c $ the leaf in $ \{c,d\} $. Then $ b $ and $ d $ are vertices in $ T $. If $ b = d $ the triangles have two common points and cannot cross. In the case $ b <_x d $ we have $ a <_x c $ by $ (iv) $. And $ (iii) $ implies $ b <_x c $ and $ a <_x d $. Consequently $ a,b <_x c,d $ and the triangles do not cross. The same argument works in the case $ d <_x b $ too.\\
 
 Now let $ t = \{x,a,b\} \in \F $ and $ u = \{P_i,Q_{i,j},Q_{i,j+1}\} \in \B $. $ t $ and $ u $ cross if and only if there are elements $ y,z \in t \backslash u $ and $ v,w \in u \backslash t $ such that the lines $ \overline{yz} $ and $ \overline{vw} $ cross. By $ (iii) $ we have $ u \subset [P_{i-1},P_{i+1}] $. Thus the fact that $ \overline{yz} $ crosses $ \overline{vw} $ implies that $ (P_{i-1},P_{i+1}) $ contains an element of $ t \backslash u $. \\
 
 Case 1: $ x \in (P_{i-1}, P_{i+1}) $. Then $ i = 0 $ or $ i = s $. But in the former case we have $ j > 0 $ and thus $ Q_{i,j} \neq P_{i-1} = x-1 $. This implies $ u \subset (x,P_{i+1}] = [P_i,P_{i+1}] $. Then $ t $ and $ u $ can only cross if $ a \in (P_i,P_{i+1}) $ or $ b \in (P_i,P_{i+1}) $ and we are in case $ 2 $. The case $ i = s $ can be similarly discounted.
 
 Case 2: $ a \in (P_{i-1}, P_{i+1}) \backslash u $ or $ b \in (P_{i-1},P_{i+1}) \backslash u $. We can assume $ b \in (P_{i-1}, P_{i+1}) \backslash u $ without loss of generality. Then $ b $ is a leaf and $ a $ is a vertex in $ T $. By $ (iii) $ we have $ a \in \{P_{i-1},P_i,P_{i+1} \} $. \\
 First assume $ a = P_{i-1} $. If $ j > 0 $, then $ a <_x P_i, Q_{i,j}, Q_{i,j+1} $. And by $ (iii) $ and $ (iv) $ we also have $ b <_x P_i, Q_{i,j}, Q_{i,j+1} $. Thus the triangles do not cross. If $ j = 0 $ then $ a = Q_{i,j} $. But we still have $ b <_x P_i, Q_{i,j+1} $ and the triangles do not cross. \\
 The case $ a = P_{i+1} $ is symmetrical to $ a = P_{i-1} $. \\
 Finally assume $ a = P_i $. Then $ b = Q_{i,k} $ for some $ k \in \{1,...,r_i - 1 \} $ and $ b \notin u $ implies $ k \notin \{j,j+1\} $. Hence $ k < j $ or $ k > j+1 $ and therefore $ b <_x Q_{i,j}, Q_{i,j+1} $ or $ b >_x Q_{i,j}, Q_{i,j+1} $. In neither case do the triangles cross. \\
 
 It remains to show that two triangles $ t = \{P_i, Q_{i,j}, Q_{i,j+1}\}, u = \{P_k, Q_{k,l}, Q_{k,l+1}\} \in \B $ do not cross. If $ i \neq k $, then we can assume $ i < k $. In this case $ (iii) $ and $ (iv) $ imply $ P_i, Q_{i,j} \leq_x P_k, Q_{k,l+1} $. If $ Q_{i,j+1} $ is a leaf or $ i < k-1 $, then also $ Q_{i,j+1} \leq_x P_k, Q_{k,l+1} $. Otherwise $ Q_{i,j+1} = P_k $ and thus $ Q_{i,j+1} \in t \cap u $. Similarly we have $ P_i, Q_{i,j} <_x Q_{k,l} $ or $ Q_{k,l} \in t \cap u $. Hence $ t \backslash u <_x u \backslash t $ and the triangles do not cross. \\
 Therefore we can assume $ i = k $. If $ j = l $ the triangles are one and the same and do not cross. We can assume $ j < l $. Then $ Q_{i,j}, Q_{i,j+1} \leq_x Q_{k,l},Q_{k,l+1} $ and again, the triangles do not cross. \\
 
 We have shown that $ \F \cup \B $ is a weakly separated family of triangles. Therefore there is a maximal weakly separated family $ \widetilde{\F} $ containing $ \F \cup \B $. Set $ \widetilde{G} := G(\widetilde{\F}_x) $. Clearly $ G $ is a subgraph of $ \widetilde{G} $. \\
 
 Step 2: Show $ G = \widetilde{G} $. \\
 Assume there is a vertex $ \widetilde{P} $ in $ \widetilde{G} $ that is not in $ G $. $ \widetilde{P} $ has degree $ \geq 1 $ by definition of the graph. Therefore there is a triangle $ \{x,\widetilde{P}, \widetilde{Q}\} \in \widetilde{\F} $. Let $ i $ be maximal with $ P_{i-1} <_x \widetilde{P} $. Then $ \widetilde{P} \in (P_{i-1},P_i) $. Now $ \overline{x\widetilde{P}} $ crosses $ \{P_i,Q_{i,0},Q_{i,1}\} = \{P_i,P_{i-1},Q_{i,1}\} \in \B $ in $ \overline{P_{i-1},P_i} $ unless $ \widetilde{Q} \in \{P_{i-1},P_i\} $.\\
 
 Case 1: $ Q_{i,1} <_x \widetilde{P} $. \\
 Let $ j>0 $ be maximal with $ Q_{i,j} <_x \widetilde{P} $. Then $ \overline{x\widetilde{P}} $ crosses $ \{P_i,Q_{i,j},Q_{i,j+1}\} \in \B $ in $ \overline{Q_{i,j}Q_{i,j+1}} $, unless $ \widetilde{Q} \in \{Q_{i,j},Q_{i,j+1}\} $. But $ \{P_{i-1},P_i\} \cap \{Q_{i,j},Q_{i,j+1}\} = \emptyset $ since $ j > 0 $. Hence $ \widetilde{Q} $ cannot fulfill both conditions and $ \{x,\widetilde{P},\widetilde{Q}\} $ crosses a triangle in $ \B $. Contradiction to $ \widetilde{\F} $ being weakly separated. \\
 
 Case 2: $ \widetilde{P} <_x Q_{i-1,r_{i-1}-1} $. \\
 This case is symmetrical to case 1. \\
 
 Case 3: $ Q_{i-1,r_{i-1}-1} <_x \widetilde{P} <_x Q_{i,1} $. \\
 Then $ \overline{x\widetilde{P}} $ crosses $ \{P_i,Q_{i,0},Q_{i,1}\} = \{P_i,P_{i-1},Q_{i,1}\} \in \B $ in $ \overline{P_{i-1}P_i} $ and $ \overline{P_{i-1}Q_{i,1}} $ unless $ \widetilde{Q} \in \{P_{i-1}, P_i\} $ and $ \widetilde{Q} \in \{P_{i-1}, Q_{i,1}\} $. This implies $ \widetilde{Q} = P_{i-1} $. \\
 Similarly $ \overline{x\widetilde{P}} $ crosses $ \{P_{i-1},Q_{i-1,r_{i-1}-1},Q_{i-1,r_{i-1}}\} = \{P_{i-1},Q_{i-1,r_{i-1}-1},P_i\} \in \B $ in $ \overline{Q_{i-1,r_{i-1}-1}P_i} $, unless $ \widetilde{Q} = Q_{i-1,r_{i-1}-1} $ or $ \widetilde{Q} = P_i $. Contradiction to $ \widetilde{Q} = P_{i-1} $. \\

 The only thing left is to prove that every edge in $ \widetilde{G} $ is also in $ G $. Let $ \{\widetilde{P}, \widetilde{Q} \} $ be an edge in $ \widetilde{G} $ that is not in $ G $. Assume that $ \widetilde{P}$ and $ \widetilde{Q} $ are vertices in $ T $. Since $ T $ is a triangulation, no new edge can be added without crossing an edge $ \{P_i, P_j \} $ in $ T $. But then $ \{x,\widetilde{P}, \widetilde{Q} \} $ would cross $ \{x,P_i,P_j\} $. Contradiction to $ \widetilde{\F} $ being weakly separated. \\
 Therefore we can assume $ \widetilde{Q} = Q_{i,j} $ for some $ i \in \{0,...,s\} $ and $ j \in \{1,...,r_i-1\} $. If $ \widetilde{P} = P_i $ then the edge is also in $ G $ and we are finished. Otherwise, if $ \widetilde{P} \notin [Q_{i,j-1}, Q_{i, j+1}] $ then $ \{x,\widetilde{P}, Q_{i,j}\}$ crosses $ \{x,P_i,Q_{i,j-1}\} $ or $ \{x,P_i,Q_{i,j+1}\} $. The only cases left to consider are thus $ \widetilde{P} = Q_{i,j-1} $ and $ \widetilde{P} = Q_{i,j+1} $. \\
 Assume $ \widetilde{P} = Q_{i,j-1} $. If $ P_i <_x Q_{i,j-1} $ then $ \{x, Q_{i,j-1}, Q_{i,j} \} $ crosses $ \{P_i, Q_{i,j}, Q_{i,j+1} \} \in \B $, since $ \overline{xQ_{i,j-1}} $ crosses $ \overline{P_iQ_{i,j+1}} $ in that case. Contradiction. \\
 If $ Q_{i,j} <_x P_i $ and $ j \geq 2 $ then $ \{x, Q_{i,j-1}, Q_{i,j} \} $ crosses $ \{ P_i, Q_{i,j-2}, Q_{i,j-1} \} \in \B $ in $ \overline{xQ_{i,j}} $ and $ \overline{Q_{i,j-2}P_i} $. If $ j = 1 $ then $ \{x, Q_{i,j-1}, Q_{i,j}\} $ crosses $ \{P_{i-1}, Q_{i-1,r_i-1}, P_i \} \in \B $ instead. Contradiction. \\
 If $ Q_{i,j-1} <_x P_i <_x Q_{i,j} $ then $ \{x, Q_{i,j-1}, Q_{i,j} \} $ crosses $ \{x, P_i, P_{i+1} \} $ in $ \overline{Q_{i,j-1} Q_{i,j}} $ and $ \overline{P_i P_{i+1}} $. Hence $ \widetilde{P} \neq Q_{i,j-1} $. $ \widetilde{P} \neq Q_{i,j+1} $ follows by symmetry. Contradiction.\\
 We have shown that every edge in $ \widetilde{G} $ is also in $ G $ and thus $ G = \widetilde{G} $. This proves the Theorem.
\end{proof}

\section{Application to Grassmannians}

The Grassmannian $ \G(k,n) $ (over $ \C $) is the set of $ k $-dimensional linear subspaces of $ \C^n $. This set can be embedded as a closed subvariety into projective space using the Plücker embedding:
\begin{align*}
\G(k,n) &\to \pr(\bigwedge\nolimits^k \C^n) \\
\langle v_1,\dots,v_k\rangle &\mapsto [v_1 \wedge \dots \wedge v_k]
\end{align*}
The vector space $ \bigwedge^k \C^n $ has a basis consisting of the vectors $ e_{i_1} \wedge \dots \wedge e_{i_k} $ where $ \{e_1, \dots, e_n\} $ is the standard basis of $ \C^n $ and $ 1 \leq i_1 < \dots < i_k \leq n $. Then the homogeneous coordinate ring $ \C[\pr(\bigwedge^k \C^n)] $ is generated by the homogeneous coordinates $ x_{i_1,\dots,i_k} $. The homogeneous coordinate ring $ \C[\G(k,n)] $ of the Grassmannian is the quotient with the ideal $ I $ generated by the so called Plücker relations 
\[
\Sigma_{l=0}^k (-1)^l x_{i_1,\dots,i_{k-1},j_l} x_{j_0, \dots, \hat{j_l}, \dots, j_k}
\]
for arbitrary $ 1 \leq i_1 < \dots < i_{k-1} \leq n $ and $ 1 \leq j_0 < \dots < j_n \leq n $.
\begin{defi}
We will denote the images of the linear coordinates $ x_{i_1,...,i_k} $ under this quotient by $ \Delta^{i_1,...,i_k} $ and call them the Plücker coordinates.
\end{defi}

It has been shown that $ \C[\G(k,n)] $ has the structure of a cluster algebra. In the case $ k = 2 $ this was proved by Fomin-Zelevinsky \cite{Fomin_2003}, the general case by Scott \cite{scott2006grassmannians}. All Plücker variables are cluster variables, but in general there are other cluster variables too. The frozen cluster variables of $ \C[\G(k,n)] $ are exactly the continuous Plücker variables $ \Delta^{i,i+1,\dots,i+k-1} $. The following Theorem by Oh, Postnikov and Speyer establishes when a family of Plücker variables is a cluster in $ \G(k,n) $:
\begin{theorem}[\cite{oh2015weak}]
    Let $ C $ be a collection of Plücker variables in $ \G(k,n) $. The following are equivalent:
    \begin{itemize}
        \item[(i)] $ C $ is a cluster.
        \item[(ii)] The set of index sets of $ C $ is maximal weakly separated.
    \end{itemize}
\end{theorem}
Therefore we will identify clusters consisting only of Plücker variables with the corresponding maximal weakly separated families of triangles $ \F $. \\

In the following we consider only the case $ k = 3 $. We also assume that a point $ P \in \G(k,n) $ in the Grassmannian is given. Let $ v:\C[\G(k,n)] \to \C $ be the evaluation map at that point.

\begin{rem}
 Unlike with a triangle $ t = \{a,b,c\} $, the order of the variables matters for Plücker-coordinates, we have
 \[
  \Delta^{a b c} = sign(\sigma) \Delta^{\sigma(a) \sigma(b) \sigma(c)}
 \]
 for any permutation $ \sigma \in S(\{a,b,c\}) $. In the following we therefore write $ v(\{a,b,c\}) := v(\Delta^{o(a,b,c)}) $ where $ o $ is a function ordering the indices.

\end{rem}

\begin{defi}[\cite{baur2021friezes}]
Plücker variables of the form $ \Delta^{i,i+1,j} $ for $ i,j \in [n] $ with $ j \neq i+2 $ are called almost continuous Plücker variables.
\end{defi}

We are particularly interested in almost continuous Plücker variables of the form $ \Delta^{i,i+1,i+3} $ and $ \Delta^{i,i+2,i+3} $. Choose $ x \in [n] $. We will see that $ \F_x $ is closely related to two such Plücker variables: $ \Delta^{x-2,x-1,x+1} $ and $ \Delta^{x-1,x+1,x+2} $. We call these the almost continuous Plücker variables in $ x $. To calculate these (or equivalently $ v(t) $ of the corresponding triangles) we need an assumption on $ v $.

\begin{defi}
Let $ \F $ be some cluster of Plücker variables in $ \G(k,n) $ and $ x \in [n] $. If $ v(t) = \lambda $ for all $ t \in \F_x $ and some $ 0 \neq \lambda \in \C $ we call $ (\F,v) $ unitary in $ x $.
\end{defi}

We assume that a family $ \F $ exists such that $ (\F,v) $ is unitary in $ x $. Since $ P $ is a point in projective space we can then assume $ v(t) = 1 $ for all $ t \in \F_x $. We will now calculate $ v(\{x-1,x+1,x+2\}) $ and $ v(\{x-2,x-1,x+1\}) $. For that we need the following mutation of maximal weakly separated families of triangles (and thus of clusters in $ \G(k,n) $).

\begin{lem}[\cite{scott2005quasi} Theorem 3] \label{lem:mutation}
    Let $ \F $ be a maximal weakly separated family of triangles. Let $ z, a, b, c, d \in [n] $ be $ 5 $ pairwise distinct points where $ a, b, c, d $ are cyclically ordered. Assume that the triangles $ \{z,a,b\}, \{z,b,c\},\\ \{z,c,d\}, \{z,d,a\}, \{z,a,c\} $ are all elements of $ \F $. Then $ \{z,b,d\} \notin \F $ and
    \[
    \F^\prime := \F \cup \{\{z,b,d\}\} \backslash \{\{z,a,c\}\}
    \]
    is also a maximal weakly separated family of triangles.
\end{lem}

In the following we will describe a sequence $ \F, \F^1, ..., \F^r $ of clusters that arise by mutation such that $ \{x-1,x+1,x+2\} $ and $ \{x-2,x-1,x+1\} $ are border triangles in $ \F^r_x $ and such that $ (\F^i,v) $ is unitary in $ x $ for all $ i $. Moreover, if the values in all border triangles of $ \F_x $ are known, we will see how to calculate the values of the border triangles in $ \F^i_x $. For that we need a specific class of Plücker relations:

\begin{lem}[3-term Plücker relations]\label{lem:plücker}
 Let $ z,a,b,c,d \in [n] $ be pairwise distinct with $ \overline{ac} $ and $ \overline{bd} $ crossing. Then
 \[
    \Delta^{o(z,a,c)} \cdot \Delta^{o(z,b,d)} = \Delta^{o(z,a,b)} \cdot \Delta^{o(z,c,d)} + \Delta^{o(z,a,d)} \cdot \Delta^{o(z,b,c)}.
 \]
 Here $ o $ is a function that orders the indices.
\end{lem}
\begin{proof}
See \cite{scott2006grassmannians} p.5 (1) with $ I = \{z\} $.
\end{proof}

First we will remove all leaves in $ G(\F_x) $ except for $ x-2 $ and $ x+2 $ if they are leaves. Because they correspond to the frozen triangles $ \{x,x+1,x+2\} $ and $ \{x-2,x-1,x\} $ that cannot be removed by mutation.

\begin{lem}[Removing leaves]\label{lem:leaf_removal}
 Let $ (\F,v) $ be unitary in $ x $. Let $ P $ be a triangulation point in $ G(\F_x) $. Let $ Q_2 $ be a leaf in $ P $ and let $ Q_1 $ and $ Q_3 $ either be the previous and next leaves or the previous and next triangulation points, if no such leaves exist. Also assume $ Q_2 \notin \{x-2,x+2\} $. Then there is a mutation to the cluster $ \F^\prime := \F \cup \{\{P,Q_1,Q_3\}\} \backslash \{ \{x,P,Q_2\} \} $. Moreover
 \[
v(\{P,Q_1,Q_3\}) = v(\{P,Q_1,Q_2\}) + v(\{P,Q_2,Q_3\}).
 \] 
\end{lem}
\begin{proof}
$ Q_1, Q_2, Q_3 $ are all either leaves in $ P $ or adjacent triangulation points, therefore $ \{x,P,Q_1\}, \{x,P,Q_2\}, \{x,P,Q_3\} \in \F $. Moreover $ \{P,Q_1,Q_2\}, \{P,Q_2,Q_3\} \in \F $ by Lemma \ref{lem:border_triangles}. By setting $ z = P, a = x, b = Q_1, c = Q_2, d = Q_3 $ we are in the situation of Lemma \ref{lem:mutation}. Therefore
\[
\F \cup \{\{P,Q_1,Q_3\}\} \backslash \{\{x,P,Q_2\}\}
\]
is a maximal weakly separated family and thus a cluster as claimed. Also $ \overline{xQ_2} $ and $ \overline{Q_1Q_3} $ are crossing and Lemma \ref{lem:plücker} implies
\[
\Delta^{o(P,x,Q_2)} \cdot \Delta^{o(P,Q_1,Q_3)} = \Delta^{o(P,x,Q_1)} \cdot \Delta^{o(P,Q_2,Q_3)} + \Delta^{o(P,x,Q_3)} \cdot \Delta^{o(P,Q_1,Q_2)}.
\]
We apply $ v $ on both sides. Since $ (\F,v) $ is unitary in $ x $ we have $ v(\{x,P,Q_2\}) = v(\{x,P,Q_1\}) = v(\{x,P,Q_3\}) \neq 0 $ and can divide by this value. This results in
\[
v(\{P,Q_1,Q_3\}) = v(\{P,Q_2,Q_3\}) + v(\{P,Q_1,Q_2\})
\]
and we are done.
\end{proof}

We can now remove all leaves except $ x+2 $ and $ x-2 $ if they are leaves. To progress further we use the following well known fact:

\begin{lem}
    Let $ T $ be a triangulation of an $ n $-gon with $ n \geq 4 $. Then there are at least two nodes with degree $ 2 $ in $ T $ and they are not adjacent.
\end{lem}

This implies that after removing the leaves, $ G(\F_x) $ will have at least two triangulation points of degree $ 2 $, unless one of these triangulation points is $ x+1 $ or $ x-1 $. Because in that case they may still have a leaf in $ x+2 $ or $ x-2 $. \\
However $ x+1 $ and $ x-1 $ are adjacent and cannot both have degree $ 2 $ in the triangulation, therefore $ G(\F_x) $ contains at least one triangulation point $ P_i $ of degree $ 2 $ with $ P_i \notin \{x+1, x-1\} $.

\begin{lem}[Removing order 2 nodes]\label{lem:order_2_node_removal}
\leavevmode \\
Let $ (\F, v) $ be unitary in $ x $. Let $ P_i $ be a triangulation point in $ \F_x $ of order $ 2 $ with $ P_i \notin \{ x-1, x+1 \} $. Let $ P_{i-1} $ and $ P_{i+1} $ be the previous and next triangulation point. Let $ Q_1 $ be the last leaf in $ P_{i-1} $ or the previous triangulation point, if no such leaf exists. Let $ Q_2 $ be the first leaf in $ P_{i+1} $ or the next triangulation point if no such leaf exists. \\
If $ P_i \neq x+2 $ then there is a mutation to the cluster
\[\F^\prime := \F \cup \{\{Q_1, P_{i-1}, P_{i+1}\} \} \backslash \{\{x, P_{i-1}, P_i\}\}.
\] \\
Similarly if $ P_i \neq x-2 $ then there is a mutation to the cluster
\[
\F^\prime := \F \cup \{\{P_{i-1}, P_{i+1}, Q_2\} \} \backslash \{\{x, P_i, P_{i+1}\}\}.
\]
Furthermore in the first case we have
\[
v(\{Q_1, P_{i-1}, P_{i+1}\}) = v(\{Q_1, P_{i-1}, P_i\}) + v(\{P_{i-1}, P_i, P_{i+1}\})
\]
and in the second case
\[
v(\{P_{i-1}, P_{i+1}, Q_2\}) = v(\{P_i, P_{i+1}, Q_2\}) + v(\{P_{i-1}, P_i, P_{i+1}\}).
\]
\end{lem}

\begin{proof}
We focus on the case $ P_i \neq x+2 $ since the case $ P_i \neq x-2 $ is symmetrical. We have $ \{x,Q_1,P_{i-1}\} \in \F $ since $ Q_1 $ is either a leaf in $ P_{i-1} $ or the previous triangulation point. We also have $ \{x,P_{i-1},P_i\} \in \F $ since $ P_{i-1} $ and $ P_i $ are adjacent triangulation points. And the fact that $ P_i $ has order two implies $ \{x,P_{i-1},P_{i+1}\} \in \F $ because otherwise the region next to $ P_i $ in the triangulation would not be a triangle. Additionally Lemma \ref{lem:border_triangles} shows $ \{P_{i-1},Q_1,P_i\}, \{P_{i-1},P_i,P_{i+1}\} \in \F $. We set $ z = P_{i-1}, a = x, b = Q_1, c = P_i, d = P_{i+1} $. Here $ (a,b,c,d) $ are cyclically ordered since $ P_i \neq x+2 $. Lemma \ref{lem:mutation} implies that
\[
\F \cup \{\{P_{i-1},Q_1,P_{i+1}\}\} \backslash \{\{P_{i-1},x,P_i\}\}
\]
is a maximal weakly separated family and thus a cluster of Plücker variables. Moreover $ \overline{xP_i} $ and $ \overline{Q_1P_{i+1}} $ are crossing. Therefore Lemma \ref{lem:plücker} yields
\begin{align*}
\Delta^{o(P_{i-1},x,P_i)} \cdot \Delta^{o(P_{i-1},Q_1,P_{i+1})} &= \Delta^{o(P_{i-1},x,Q_1)} \cdot \Delta^{o(P_{i-1},P_i,P_{i+1})} \\
&+ \Delta^{o(P_{i-1},x,P_{i+1})} \cdot \Delta^{o(P_{i-1},Q_1,P_i)}.
\end{align*}
By applying $ v $ and using  the fact that $ (\F,v) $ is unitary in $ x $ we get
\[
v(\{P_{i-1},Q_1,P_{i+1}\}) = v(\{P_{i-1},P_i,P_{i+1}\}) + v(\{P_{i-1},Q_1,P_i\})
\]
as claimed.
\end{proof}

We can now formulate an algorithm. In it we denote the triangulation points in $ G(\F_x) $ with $ P_1,\dots,P_s $ and the leaves in $ P_i $ with $ Q_{i,1},\dots,Q_{i,r_i-1} $. We write $ Q_{i,0} := P_{i-1} $ and $ Q_{i,r_i} := P_{i+1} $. The algorithm assumes $ v(t) = 1 $ for all $ t \in \F_x $.

\begin{algorithm}[H] \label{alg:quiddity}
\KwData{Graph $ G(\F_x) $ and the valuations of all border triangles}
\ForEach{triangulation point $ P_i \notin \{x+1,x-1\} $}{
    Remove all leaves in $ P_i $ \\
    $ label(P_i) := \Sigma_{j=0}^{r_i-1}v(\{P_i,Q_{i,j},Q_{i,j+1}\})$
}
\If{$ x+2 $ is a leaf}{
    Remove all leaves in $ x + 1 $ except $ x+2 $ \\
    $ label(x+1) := \Sigma_{j=1}^{r_1-1} v(\{x+1,Q_{1,j},Q_{1,j+1}\}) $
}
\If{$ x-2 $ is a leaf}{
    Remove all leaves in $ x + 1 $ except $ x-2 $ \\
    $ label(x-1) := \Sigma_{j=0}^{r_s-2} v(\{x-1,Q_{s,j},Q_{s,j+1}\}) $
}
\While{ $ G(\F_x) $ has more than $ 3 $ edges}{
    Find a triangulation point $ P_i \notin \{x+1,x-1\} $ with degree $ 2 $ \\
    \uIf{$ P_i = x+2 $}{
        $ label(x+1) := label(x+2) $ \\
        Remove the edge to the other neighbour $ P_j $ of $ x+2 $ \\
        $ label(P_j) := label(P_j) + label(x+2) $ \\
        Remove the label at $ x+2 $
    }
    \uElseIf{$ P_i = x-2 $}{
        $ label(x-1) := label(x-2) $ \\
        Remove the edge to the other neighbour $ P_j $ of $ x-2 $ \\
        $ label(P_j) := label(P_j) + label(x-2) $ \\
        Remove the label at $ x-2 $
    }
    \Else{
        Let $ P_j $ and $ P_k $ be the two neighbours of $ P_i $ \\
        $ label(P_j) := label(P_j) + label(P_i) $ \\
        $ label(P_k) := label(P_k) + label(P_i) $ \\
        Remove $ P_i $ and its edges
    }
}
Output $ v(\{x-1,x+1,x+2\}) = label(x+1) $ \\
Output $ v(\{x-2,x-1,x+1\}) = label(x-1) $
\caption{Almost continuous Plücker variables in $ x $}
\end{algorithm}

\begin{proof}
   The label in $ P_i $ is equal to $ v(\{P_{i-1},P_i,P_{i+1}\}) $ for $ i \notin \{1,s\} $. The label in $ x+1 $ is $ v(\{x+1, x+2, P_2\}) $ if $ x+2 $ is a leaf. Similarly the label in $ x-1 $ is $ v(\{P_{s-1},x-2,x-1\}) $ if $ x-2 $ is a leaf. By Lemma \ref{lem:leaf_removal} all of these labels are initialized to the right values. All operations after that are allowed by Lemma \ref{lem:order_2_node_removal}. In the end only the frozen edges $ \{x-2,x-1\}, \{x-1,x+1\}, \{x+1,x+2\} $ remain. Now the only border triangles are $ \{x-1,x+1,x+2\} $ and $ \{x-2,x-1,x+1\} $ which are now the labels at the only remaining triangulation points $ x+1 $ and $ x-1 $.
\end{proof}

If the cluster $ \F $ is specialized to $ 1 $, i.E. $ v(t) = 1 $ for all $ t \in \F $, we can use this method to calculate $ v(\{i,i+1,i+3\}) $ and $ \{i,i+1,i+3\} $ for all $ i \in [n] $. \\
If the values of all these  are known we can also calculate the value of all other almost continuous triangles. We use the following notation:

\begin{notation}\label{not:plücker}
We write
\[
\Delta_k(i) := v(\{i,i+1,i+k+2\})
\]
and
\[
\Delta^k(i) := v(\{i,i+k+1,i+k+2\})
\]
for the values of $ v $ at the almost continuous Plücker variables.
\end{notation}
\begin{theorem} \label{thm:plücker_recursion}
Let $ v : \C[\G(3,n)] \to \C $ be an algebra homomorphism. Assume that $ v(\Delta^{i,i+1,i+2}) = 1 $ for all continuous Plücker relations. For $ i \in [n] $ and $ 3 \leq k < n-3 $ we have the recursive formulas
\[
\Delta_k(i) = \Delta_1(i) \Delta_{k-1}(i+1) - \Delta^1(i+1) \Delta_{k-2}(i+2) + \Delta_{k-3}(i+3)
\]
and
\[
\Delta^k(i) = \Delta^1(i+k-1) \Delta^{k-1}(i) - \Delta_1(i+k-2) \Delta^{k-2}(i) + \Delta^{k-3}(i).
\]
For $ k = 2 $ we have
\[
\Delta_2(i) = \Delta_1(i)\Delta_1(i+1) - \Delta^1(i+1)
\]
and
\[
\Delta^2(i) = \Delta^1(i+1) \Delta^1(i) - \Delta_1(i).
\]
\end{theorem}

\begin{proof}
    Apply Lemma \ref{lem:plücker} with $ z = i+1, a = i, b = i+2, c = i+3 $ and $ d = i + k + 2 $. This yields
    \[
    \Delta^{i,i+1,i+3} \Delta^{i+1,i+2,i+k+2} = \Delta^{i,i+1,i+2} \Delta^{i+1,i+3,i+k+2} + \Delta^{i,i+1,i+k+2} \Delta^{i+1,i+2,i+3}.
    \]
After using $ v $ on both sides we get
\[
\Delta_1(i) \Delta_{k-1}(i+1) = 1 \cdot v(\Delta^{i+1,i+3,i+k+2}) + \Delta_k(i) \cdot 1.
\]
We write $ \Delta^1_k(i) := v(\Delta^{i,i+2,i+k+3}) $. Then
\begin{equation}\label{eq:below}
    \Delta_k(i) = \Delta_1(i) \Delta_{k-1}(i+1) - \Delta^1_{k-2}(i+1).
\end{equation}
Now apply Lemma \ref{lem:plücker} with $ z = i+2, a = i, b = i+1, c = i+3, d = i+k+3 $:
\[
\Delta^{i,i+2,i+3} \Delta^{i+1,i+2,i+k+3} = \Delta^{i,i+1,i+2} \Delta^{i+2,i+3,i+k+3} + \Delta^{i,i+2,i+k+3} \Delta^{i+1,i+2,i+3}.
\]
After applying $ v $ this becomes
\begin{equation} \label{eq:below_up}
    \Delta^1_k(i) = \Delta^1(i) \Delta_k(i+1) - \Delta_{k-1}(i+2)
\end{equation}
Now combining (\ref{eq:below}) and (\ref{eq:below_up}) yields
\[
\Delta_k(i) = \Delta_1(i) \Delta_{k-1}(i+1) -  \Delta^1(i+1) \Delta_{k-2}(i+2) + \Delta_{k-3}(i+3).
\]
For the other formula we apply Lemma \ref{lem:plücker} with $ z = i+k+1, a = i, b = i+k-1, c = i+k $ and $ d = i+k+2 $:
\begin{align*}
\Delta^{i,i+k,i+k+1} \Delta^{i+k-1,i+k+1,i+k+2} &= \Delta^{i,i+k-1,i+k+1} \Delta^{i+k,i+k+1,i+k+2}\\
&+ \Delta^{i,i+k+1,i+k+2} \Delta^{i+k-1,i+k,i+k+1}
\end{align*}
We set $ \Delta^k_1(i) := \Delta^{i,i+k+1,i+k+3} $. By using $ v $ we get
\begin{equation} \label{eq:up}
    \Delta^k(i) = \Delta^{k-1}(i) \Delta^1(i+k-1) - \Delta^{k-2}_1(i).
\end{equation}
Finally apply Lemma $ \ref{lem:plücker} $ with $ z = i+k+1, a = i, b = i+k, c = i+k+2 $ and $ d = i+k+3 $:
\begin{align*}
\Delta^{i,i+k+1,i+k+2} \Delta^{i+k,i+k+1,i+k+3} &= \Delta^{i,i+k,i+k+1} \Delta^{i+k+1,i+k+2,i+k+3}  \\
&+ \Delta^{i, i+k+1, i+k+3} \Delta^{i+k,i+k+1,i+k+2}
\end{align*}
which becomes
\begin{equation} \label{eq:up_below}
    \Delta^k_1(i) = \Delta^k(i) \Delta_1(i+k) - \Delta^{k-1}(i)
\end{equation}
after using $ v $. Together with (\ref{eq:up}) we get
\[
\Delta^k(i) = \Delta^1(i+k-1) \Delta^{k-1}(i) - \Delta_1(i+k-2) \Delta^{k-2}(i) + \Delta^{k-3}(i).
\]
The formulas for $ k = 2 $ follow immediately from (\ref{eq:below}) and (\ref{eq:up}), because
\[
\Delta_0^1(i+1) = v(\Delta^{i+1,i+3,i+4}) = \Delta^1(i+1)
\]
and
\[
\Delta_1^0(i) = v(\Delta^{i,i+1,i+3}) = \Delta_1(i).
\]
\end{proof}

\section{Application to tame $ \SL_3 $-frieze patterns}

\begin{defi}[\cite{cordes1972generalized}]
Let $ w \in \N_{>0} $ and let $ R $ be an integral domain. Consider the following infinite array
\[
\xymatrix@C=0.06em@R=.09em{
&&&&&&&& \dots && 0 && 0 && 0 && \dots \\
&&&&&&& \dots && 0 && 0 && 0 && \dots \\
&&&&&& \dots && 1 && 1 && 1 && \dots \\
&&&&& \dots && m_{0,w} && m_{1,w} && m_{2,w} && \dots \\
&&&&&&&& \iddots \\
&&& \dots && m_{0,1} && m_{1,1} && m_{2,1} && \dots \\
&& \dots && 1 && 1 && 1 && \dots \\
& \dots && 0 && 0 && 0 && \dots \\
 \dots && 0 && 0 && 0 && \dots
}
\]
where the $ 2 $ highest and lowest rows only contain $ 0 $, the next rows after that contain only $ 1 $ and all other entries $ m_{i,j} $ for $ i \in \Z $ and $ j \in \{1,\dots,w\} $ are elements of $ R $.
\begin{itemize}
    \item[(i)] The array is called a $ \SL_3 $-frieze pattern if all $ 3 \times 3 $-diamonds of entries have determinant $ 1 $. $ w $ is called the width of the frieze pattern.
    \item[(ii)] If all $ 4 \times 4 $-diamonds have determinant $ 0 $ the frieze pattern is called tame.
    \item[(iii)] If all entries are positive integers, the frieze pattern is called integral.
\end{itemize}
\end{defi}

\begin{rem}
    Let $ F $ be a tame $ \SL_3 $-frieze of width $ w $. Then $ F $ has a horizontal period of $ w + 4 $. In other words $ m_{i,j} = m_{i+w+4,j} $ for all $ i \in \Z $ and $ j \in \{1, \dots, w\} $. In particular $ F $ is then completely described by the fundamental region containing only the nontrivial elements $ m_{i,j} $ with $ i \in \{0, \dots, w+3\} $ and $ j \in \{1, \dots, w\} $.
\end{rem}
\begin{proof}
    See \cite{morier2014linear} Corollary 7.1.1.
\end{proof}

\begin{defi}
Let $ \cA $ be a cluster algebra of rank $ m $, that is all clusters in $ \cA $ have cardinality $ m $. Let $ a = (a_1,\dots,a_m) \in (\C\backslash \{ 0 \})^m $ and let $ C = (x_1, \dots, x_m) $ be a cluster in $ \cA $. Then the algebra homomorphism $ \cA \to \C $ that is determined by $ x_i \mapsto a_i $ is called the specialization of $ C $ to $ a $. If $ a = (1, \dots, 1) $ we call this the specialization of $ C $ to $ 1 $.
\end{defi}

\begin{rem}
For $ n \geq 3 $ a maximal weakly separated family of triangles in $ [n] $ has cardinality $ 3n - 8 $. In particular the cluster algebra $ \C[\G(3,n)] $ has rank $ 3n-8 $.
\end{rem}
\begin{proof}
See \cite{oh2015weak} Theorem 3.3.
\end{proof}

In \cite{baur2021friezes} it was shown that there is a tame $ \SL_3 $-frieze pattern $ \cP_{(3,n)} $ of width $ w = n - 4 $ over $ \C[\G(3,n)] $ with $ m_{i,j} = \Delta^{o(i,i-1,j+i-1)} $. Here $ o $ is a function that orders the indices and the indices are taken$ \mod \; n $. This frieze is called the Plücker frieze of type $ (3,n) $.

\begin{rem}
Let $ \phi: \C[\G(3,n)] \to \C $ be the algebra homomorphism from specializing a cluster. Applying $ \phi $ to all entries of $ \cP_{(3,n)} $ results in a tame $ \SL_3 $-frieze pattern $ \phi(\cP_{(3,n)}) $ over $ \C $ by \cite{baur2021friezes} Remark 3.11. Moreover \cite{morier2014linear} Proposition 3.2.1 shows that all tame $ \SL_3 $-frieze patterns of width $ n-4 $ have this form. And if $ \phi $ arises by specialising a cluster to $ 1 $ then \cite{baur2021friezes} Corollary 3.12 shows that $ \phi(\cP_{(3,n)}) $ is integral. Friezes of this form are called unitary.
\end{rem}

Using Notation \ref{not:plücker} with $ v = \phi $ we can write the fundamental region of $ \phi(\cP_{(3,n)}) $ as
\[
\xymatrix@C=0.06em@R=.09em{
&&&& \Delta^1(1) && \Delta^1(2) && \dots && \Delta^1(n)\\
&&& \Delta^2(1) && \Delta^2(2) && \dots && \Delta^2(n)\\
&&&&&& \iddots \\
& \Delta_2(1) && \Delta_2(2) && \dots && \Delta_2(n)\\
\Delta_1(1) && \Delta_1(2) && \dots && \Delta_1(n)\\
}
\]

Now Theorem \ref{thm:plücker_recursion} allows us to easily calculate all other entries if the first and last row are known. If the frieze arises from specializing a cluster of Plücker variables to $ 1 $ then Algorithm \ref{alg:quiddity} can be used to calculate the first and last row from the family $ \F $ that corresponds to the cluster.

\newpage

\bibliographystyle{alphaurl}
\newcommand{\etalchar}[1]{$^{#1}$}

\end{document}